\documentclass[12pt,twoside,reqno,psamsfonts]{amsart}

\usepackage[OT1]{fontenc}    %
\usepackage{type1cm}         %
\usepackage{amssymb}
\usepackage[dvips]{graphicx} 
\usepackage{psfrag}          
\usepackage{geometry}        
\usepackage{version}         

\numberwithin{equation}{section}

\theoremstyle{plain}
\newtheorem{theorem}{Theorem}[section]

\newtheorem{proposition}[theorem]{Proposition}
\newtheorem{lemma}[theorem]{Lemma}

\theoremstyle{remark}
\newtheorem{remark}[theorem]{Remark}

\newtheorem{example}[theorem]{Example}

\theoremstyle{definition}

\newtheorem{question}[theorem]{Question}

\renewcommand{\atop}[2]{\genfrac{}{}{0pt}{}{#1}{#2}}

\newcommand{\HH}{\mathcal{H}}

\newcommand{\QQ}{\mathcal{Q}}
\newcommand{\BB}{\mathcal{B}}
\newcommand{\R}{\mathbb{R}}

\newcommand{\Z}{\mathbb{Z}}
\newcommand{\N}{\mathbb{N}}

\newcommand{\iii}{\mathtt{i}}
\newcommand{\jjj}{\mathtt{j}}

\newcommand{\eps}{\varepsilon}

\DeclareMathOperator{\dimh}{dim_H}

\DeclareMathOperator{\dist}{dist}
\DeclareMathOperator{\diam}{diam}
\DeclareMathOperator{\proj}{proj}

\DeclareMathOperator{\vol}{vol}

\DeclareMathOperator{\spt}{spt}

\begin{document}

\title{Upper Conical density results for general measures on $\R^n$}

\author[M.\ Cs\"ornyei]{Marianna Cs\"ornyei}
\address{Department of Mathematics \\
         University College London \\
         Gower Street \\
         London WC1E 6BT \\
         United Kingdom}
\email{mari@math.ucl.ac.uk}

\author[A.\ K\"aenm\"aki]{Antti K\"aenm\"aki}
\author[T.\ Rajala]{Tapio Rajala}
\author[V.\ Suomala]{Ville Suomala}
\address{Department of Mathematics and Statistics \\
         P.O. Box 35 (MaD) \\
         FI-40014 University of Jyv\"askyl\"a \\
         Finland}
\email{antakae@maths.jyu.fi}
\email{tamaraja@maths.jyu.fi}
\email{visuomal@maths.jyu.fi}

\dedicatory{Dedicated to Professor Pertti Mattila on the occasion of his 60th birthday}
\thanks{AK acknowledges the support of the Academy of Finland (project \#114821) and TR is indebted to the Ville, Yrj\"o and Kalle V\"ais\"al\"a Fund.}
\date{\today}

\begin{abstract}
  We study conical density properties of general Borel measures on Euclidean spaces. Our results are analogous to the previously known result on the upper density properties of Hausdorff and packing type measures.
\end{abstract}

\maketitle

\section{Introduction}

The extensive study of upper conical density properties for Hausdorff
measures was pioneered by Besicovitch who studied the
conical density properties of purely $1$-unrectifiable fractals on the
plane. Since Besicovitch's time upper density
results have played
an important role in geometric measure theory. Due to the works of
Marstrand \cite{Marstrand1954}, Salli \cite{Salli1985}, Mattila
\cite{Mattila1988}, and others, the
upper conical density
properties of Hausdorff measures $\HH^s$ for all values of $0\leq
s\leq n$ are very well understood. There are also analogous results for many
(generalised) Hausdorff and packing measures, see
\cite{KaenmakiSuomala2008} and references therein.
Conical density results are useful since they give
information on the distribution of the measure if the values of the
measure are known on some small balls. The main applications deal with
rectifiability, see \cite{Mattila1995}, but often upper conical density
theorems may also be viewed as some kind of anti-porosity theorems. See
\cite{Mattila1988} and \cite{KaenmakiSuomala2008} for more on this topic.

When working with a Hausdorff or packing type measure $\mu$, it is useful to
study densities such as
\[\limsup_{r\downarrow0}\mu\bigl(X(x,r,V,\alpha)\bigr)/h(2r)\]
where $h$ is the gauge function used to construct the measure $\mu$
and $X(x,r,V,\alpha)$ is a cone around the point $x$ (see
\S2 below for the formal definition). However, most measures are so
unevenly distributed that there are no gauge functions that could be
used to approximate the measure in small balls. This is certainly the case for many self-similar and multifractal-type measures.
For these measures the above quoted results give no information. To
obtain conical density results for general measures
it seems natural to replace the value of the gauge $h$ in the
denominator by the measure of the ball $B(x,r)$ and consider upper
densities such as
\[\limsup_{r\downarrow0}\mu\bigl(X(x,r,V,\alpha)\bigr)/\mu\bigl(B(x,r)\bigr).\]
Our purpose in this paper is to study densities of this and more
general type for locally finite Borel regular measures on
$\R^n$. In particular, we will answer some of the problems posed in
\cite{KaenmakiSuomala2008}.

The paper is organised as follows. In Section \ref{sec:not}, we
setup some notation and discuss auxiliary results that will be
needed later on. In particular, we recall a dimension estimate for
average homogeneous measures obtained by E.\ J\"arvenp\"a\"a and M.\
J\"arvenp\"a\"a in \cite{JarvenpaaJarvenpaa2005}.
In Section \ref{sec:gen}, we prove an upper density result valid
for all locally finite Borel regular measures on $\R^n$. The result gives a positive answer to \cite[Question 4.3]{KaenmakiSuomala2008}. It shows that
around typical points a locally finite Borel regular measure cannot be
distributed so that it lies mostly on only one one-sided cone at all
small scales. In Section \ref{sec:conical}, we obtain more detailed
information on the distribution of the measure $\mu$ provided that its
Hausdorff dimension is bounded from below. The result, Theorem
\ref{thm:dimh_positive} is analogous to the results of Mattila
\cite{Mattila1988}, and K\"aenm\"aki and Suomala
\cite{KaenmakiSuomala2004,KaenmakiSuomala2008}, obtained before for
Hausdorff and packing type measures, and it gives strong insight to
\cite[Question 4.1]{KaenmakiSuomala2008}. In Section \ref{sec:ex}, we
give a negative answer to \cite[Question 4.2]{KaenmakiSuomala2008} and
moreover, we show that Theorem \ref{thm:dimh_positive} is not valid if
we only assume that the measure is purely $m$-unrectifiable.

\section{Notation and preliminaries} \label{sec:not}

We start by introducing some notation. Let $n \in \N$, $m \in \{ 0,\ldots,n-1 \}$, and $G(n,n-m)$ denote the space of all $(n-m)$-dimensional linear subspaces of $\R^n$. The unit sphere of $\R^n$ is denoted by $S^{n-1}$. For $x \in \R^n$, $\theta \in S^{n-1}$, $0 \le \alpha \le 1$, and $V \in G(n,n-m)$, we set
\begin{align*}
  H(x,\theta,\alpha) &= \{ y \in \R^n : (y-x) \cdot \theta > \alpha
  |y-x| \}, \\
  X^+(x,\theta,\alpha) &= H\bigl( x,\theta,(1-\alpha^2)^{1/2} \bigr), \\
  X(x,V,\alpha) &= \{ y \in \R^n : \dist(y-x,V) < \alpha|y-x| \}.
\end{align*}
We also denote $X^{+}(x,r,\theta,\alpha) = B(x,r) \cap
X^{+}(x,\theta,\alpha)$ and $X(x,r,V,\alpha) = B(x,r) \cap
X(x,V,\alpha)$, where $B(x,r)$ is the closed ball centred at $x$ with
radius $r > 0$. Observe that $X^+(x,\theta,\alpha)$ is the one side of
the two-sided cone $X(x,\ell,\alpha)$ where $\ell \in G(n,1)$ is the
line pointing to the direction $\theta$. We usually use the ``$X$ notation'' for very narrow cones whereas the ``$H$ cones'' are considered as ``almost
half-spaces''. If $V \in G(n,n-m)$, we denote the orthogonal projection
onto $V$ by $\proj_V$. Furthermore, if $B = B(x,r)$ and $t>0$, then with the notation $tB$, we mean the ball $B(x,tr)$.

By a measure we will always mean a finite nontrivial Borel regular (outer)
measure defined on all subsets of some Euclidean space $\R^n$. Since
all our results are local, and valid only almost everywhere, we could
easily replace the finiteness condition by assuming that $\mu$ is
almost everywhere locally finite in the sense that
$\mu(\{x\in\R^n : \mu(B(x,r))=\infty\text{ for all }r>0\})=0$.
The support of the measure $\mu$ is denoted by $\spt(\mu)$.
The \emph{(lower) Hausdorff dimension of the measure $\mu$} is defined by
\begin{align*}
  \dimh(\mu) &= \inf\{ \dimh(A) : A \text{ is a Borel set with }
  \mu(A)>0 \},
\end{align*}
where $\dimh(A)$
denotes the Hausdorff
dimension of the set $A \subset \R^n$, see \cite[\S
10]{Falconer1997}. With the notation $\mu|_F$, we mean the restriction
of the measure $\mu$ to a set $F \subset \R^n$, defined by $\mu|_F(A)
= \mu(F \cap A)$ for $A \subset \R^n$. Notice that trivially
$\dimh(\mu) \le \dimh(\mu|_F)$ whenever $F$ is a Borel set with
$\mu(F)>0$.
We will use the notation $\HH^s$ to denote the \emph{$s$-dimensional
  Hausdorff measure} on $\R^n$. More generally, we denote by $\HH_h$
a \emph{generalised Hausdorff measure} constructed using a gauge function
$h\colon(0,r_0)\rightarrow(0,\infty)$, see \cite[\S 4.9]{Mattila1995}.

Next we will recall the definition of the average homogeneity from
\cite{JarvenpaaJarvenpaa2005}.
If $k \in \N$, then a set $Q \subset \R^n$ is called a \emph{$k$-adic
  cube} provided that $Q = [0,k^{-l})^n + k^{-l}z$ for some $l \in \N$
and $z \in \Z^n$. The collection of all $k$-adic cubes $Q \subset
[0,1)^n$ with side length $k^{-l}$ is denoted by $\QQ_k^l$. If $Q \in
\QQ_k^l$ and $t>0$, then with the notation $tQ$, we mean the cube
centred at the same point as $Q$ but with side length $tk^{-l}$.

Let $k \in \N$ and $I_k = \{ 1,\ldots,k^n \}$. If $\iii =
(i_1,\ldots,i_l) \in I_k^l$ and $i \in I_k$, then we set $\iii,i =
(i_1,\ldots,i_l,i) \in I_k^{l+1}$. Furthermore, if $\iii =
(i_1,i_2,\ldots) \in I_k^\infty := I_k^\N$ (or $\iii \in I_k^l$) and
$j \in \N$ (or $j \le l$), then $\iii|_j := (i_1,\ldots,i_j) \in
I_k^j$. For a given measure $\mu$, we will enumerate
$k$-adic cubes $Q_i \in \QQ_k^1$ so that $\mu(Q_i) \le \mu(Q_{i+1})$
whenever $i \in I_k \setminus \{ k^n \}$. Given $l \in \N$ and $\iii
\in I_k^l$, we continue inductively by enumerating the cubes
$Q_{\iii,i} \in \QQ_k^{l+1}$ with $Q_{\iii,i} \subset Q_\iii \in
\QQ_k^l$ so that $\mu(Q_{\iii,i}) \le \mu(Q_{\iii,i+1})$ whenever $i
\in I_k \setminus \{ k^n \}$. Bear in mind that this enumeration
depends, of course, on the measure. The \emph{(upper) $k$-average
  homogeneity of $\mu$ of order $i \in I_k$} is defined to be
\begin{equation*}
  \hom_k^i(\mu) = \limsup_{l \to \infty} \tfrac{k^n}{l} \sum_{j=1}^l
  \sum_{\iii \in I_k^j} \mu(Q_{\iii,i}).
\end{equation*}
For us it is essential that the Hausdorff dimension of a measure may
be bounded above in terms of homogeneity. The following result was
obtained by E.\ J\"arvenp\"a\"a and M.\ J\"arvenp\"a\"a in \cite{JarvenpaaJarvenpaa2005}.

\begin{theorem} \label{thm:jarv}
If $\mu$ is a probability measure on $[0,1)^n$ and $\hom_{k}^i(\mu)\leq k^n \eta$ for some $0\leq \eta\leq k^{-n}$, then
\[\dimh(\mu)\leq-\frac{1}{\log
  k}\left(i\eta\log\eta+
  \left(1-i\eta\right)\log\left(\frac{1-i\eta}{k^n-i}\right)\right).\]
\end{theorem}

It is well known that although most measures on $\R^n$ are
non-doubling, still ``around typical points most scales are
doubling''. This somewhat inexact statement is made quantitative in the
following lemma. We follow the convention according to which
$c=c(\ldots)$ denotes a constant that depends only on the parameters listed
inside the parentheses.

\begin{lemma} \label{thm:doubling_scales}
  If $n,k \in \N$ and $0<p<1$, then there exists a constant $c = c(n,k,p) > 0$ so that for every measure $\mu$ on $\R^n$ and for each $\gamma>0$ we have
  \begin{equation*}
    \liminf_{l \to \infty} \tfrac{1}{l}\#\bigl\{ j \in \{ 1,\ldots,l
    \} : \mu\bigl( B(x,\gamma k^{-j}) \bigr) \ge c\mu\bigl( B(x,\gamma
    k^{-j+1}) \bigr) \bigr\} \ge p
  \end{equation*}
  for $\mu$-almost every $x \in \R^n$.
\end{lemma}

\begin{proof}
  Let $c = k^{-2n/(1-p)}$, fix a measure $\mu$ on $\R^n$ and $\gamma > 0$,
  and denote $N(x,l) =
  \#\bigl\{ j \in \{ 1,\ldots,l \} : \mu\bigl( B(x,\gamma k^{-j})
  \bigr) \ge c\mu\bigl( B(x,\gamma k^{-j+1}) \bigr) \bigr\}$ for $x
  \in \R^n$ and $l \in \N$.
  Suppose that $x\in\R^n$ is a point at which
  \begin{equation*}
    \liminf_{l \to \infty} \tfrac{1}{l} N(x,l) < p.
  \end{equation*}
  Then there are
  arbitrarily large integers $l$ such that $N(x,l) < lp$. Hence
  \begin{equation*}
    \mu\bigl( B(x,\gamma k^{-l}) \bigr) < c^{(1-p)l} \mu\bigl(
    B(x,\gamma) \bigr)
  \end{equation*}
  for any such $l$ and consequently,
  \begin{align*}
    \limsup_{r \downarrow 0} \frac{\log\mu\bigl( B(x,r) \bigr)}{\log r}
    &\ge \limsup_{l \to \infty} \frac{\log c^{(1-p)l} + \log\mu\bigl(
      B(x,\gamma) \bigr)}{\log \gamma k^{-l}} \\
    &= \limsup_{l \to \infty} \frac{2n \log k^{-l}}{\log\gamma k^{-l}} = 2n>n.
  \end{align*}
  But this is possible only in a set of $\mu$-measure zero, see
  for example \cite[Proposition 10.2]{Falconer1997}. The claim thus follows.
\end{proof}

\section{A general conical density estimate}\label{sec:gen}

Our first result is a conical
density theorem valid for all measures on $\R^n$. This result is
motivated by \cite[Question 4.3]{KaenmakiSuomala2008} asking if
\[\limsup_{r\downarrow 0}\inf_{\theta\in S^{n-1}}\frac{\mu\bigl(B(x,r)\setminus
H(x,\theta,\alpha)\bigr)}{h(2r)}\geq c(n,\alpha)\limsup_{r\downarrow 0}
\frac{\mu\bigl(B(x,r)\bigr)}{h(2r)}\]
holds $\mu$-almost everywhere for all measures $\mu$ on $\R^n$ and all
doubling gauge functions $h$. We shall
formulate our result for densities having $\mu\bigl(B(x,r)\bigr)$ in the
denominator rather than $h(2r)$ because we believe that these
densities are more natural in this general setting. The original
question may also be answered in the positive by a slight modification of the
proof below.

\begin{theorem}\label{thm:all}
  If $n \in \N$ and $0 < \alpha \le 1$, then there exists a constant $c = c(n,\alpha)>0$ so that for every measure $\mu$ on $\R^n$ we have
  \begin{equation*}
    \limsup_{r \downarrow 0} \inf_{\theta \in S^{n-1}} \frac{\mu\bigl(
      B(x,r) \setminus H(x,\theta,\alpha) \bigr)}{\mu\bigl( B(x,r)
      \bigr)} > c
  \end{equation*}
  for $\mu$-almost every $x \in \R^n$.
\end{theorem}

\begin{proof}
It is enough to consider non-atomic measures since
\[\limsup_{r\downarrow 0}\inf_{\theta\in
  S^{n-1}}\frac{\mu\bigl(B(x,r)\setminus
  H(x,\theta,\alpha)\bigr)}{\mu\bigl(B(x,r)\bigr)}=1\]
if $\mu(\{x\})>0$.

Because we want to use only a finite set of directions, we cover
the set $S^{n-1}$ with cones $\{ H(0, \theta_i, \beta) \}_{i=1}^K$,
where $\beta = \cos\bigl( \arccos(\alpha/2)-\arccos(\alpha) \bigr)$
and $K=K(n,\alpha) \in \N$.  For all $\theta\in S^{n-1}$ there is
$i\in\{1,\ldots,K\}$ so that $H(x,\theta,\alpha)\subset H(x,\theta_i,\alpha/2)$
for all $x\in\R^n$. Given this it is enough to show that for all
measures $\mu$ on $\R^n$ we have
\begin{equation}\label{uusivaite}
    \limsup_{r \downarrow 0} \min_{i\in\{1,\ldots,K\}} \frac{\mu\bigl(
      B(x,r) \setminus H(x,\theta_i,\alpha/2) \bigr)}{\mu\bigl( B(x,r)
      \bigr)} > c=c(\alpha,n)>0
\end{equation}
for $\mu$-almost all $x \in \R^n$.

To prove \eqref{uusivaite} we first apply Lemma
\ref{thm:doubling_scales} to find a constant $c'
< \infty$ depending only on $n$ (choosing $c'=3^{2n}$ will do) so that for
all measures $\mu$ and for every radius $R > 0$ we have the following: For $\mu$-almost every $x \in \R^n$ there is a scale $r < R$ so that
\begin{equation*}
  \mu\bigl( B(x,3r) \bigr) \le c'\mu\bigl( B(x,r) \bigr).
\end{equation*}
We will prove that
\eqref{uusivaite} holds with $c = c(n,\alpha)=
(9c'K)^{-1}$. Assume on the contrary that this is not the case. Then
we find a non-atomic measure $\mu$ and $r_0 >
0$ so that the set
  \begin{equation*}
    A := \{ x \in \R^n : \min_{i\in\{1,\ldots,K\}}
    \frac{\mu\bigl( B(x,r) \setminus H(x, \theta_i, \alpha/2)
      \bigr)}{\mu\bigl( B(x,r) \bigr)} < 2c \text{ for every } 0 < r
    \le r_0 \}
  \end{equation*}
  has positive $\mu$-measure. Now $A$ is seen to be a Borel set by
  standard methods and thus $\mu$-almost all $z\in A$ are
  $\mu$-density points of $A$, see \cite[Corollary
  2.14]{Mattila1995}. Thus we may find a point $z \in A$ and a radius
  $0 < r_1 \le r_0/2$ so that
  \begin{equation} \label{eq:densitypoint}
    \mu\bigl( A \cap B(z,r_1) \bigr) \ge \tfrac12\mu\bigl( B(z,r_1) \bigr)
  \end{equation}
  and
  \begin{equation} \label{eq:tripling}
    \mu\bigl( B(z,3r_1) \bigr) \le c'\mu\bigl( B(z,r_1) \bigr).
  \end{equation}
  Now $A\subset \bigcup_{i=1}^{K} A_i$, where
  \begin{equation*}
    A_i := \bigl\{ x \in A\,:\,\mu\bigl( B(x,2r_1)\setminus
    H(x,\theta_i,\alpha/2)\bigr) < 2c \mu\bigl( B(x,2r_1)
    \bigr)\bigr\},
  \end{equation*}
  and thus we may find $j \in \{ 1,\ldots,K \}$ so that
  \begin{equation} \label{eq:bestdirection}
    \mu\bigl( A_j \cap B(z,r_1) \bigr) \ge K^{-1}\mu \bigl( A
    \cap B(z,r_1) \bigr).
  \end{equation}
  Next take a point $y$ from the closure of $A_j\cap B(z,r_1)$ so that it
  maximises the inner product $x\cdot\theta_j$ in the closure of $A_j\cap B(z,r_1)$.
Since the measure $\mu$ is non-atomic, there is a small radius $r_2 <
r_1$ so that
  \begin{equation} \label{eq:rightconepoint}
    \mu\bigl( B(y,r_2) \bigr) < c'c \mu\bigl( B(z,r_1) \bigr).
  \end{equation}
  Now choose any point $y' \in  A_j\cap B(z,r_1) \cap B(y,\alpha
  r_2/3)$ and cover the set $A \cap B(z,r_1)$ with sets $D_1$,
  $D_2$, and $D_3$ defined by $D_1 = B(y,r_2)$,
    $D_2 = B(y',2r_1) \setminus H(y',\theta_j,\alpha/2)$, and
    $D_3 = \bigl( A \cap B(z,r_1) \bigr) \setminus (D_1 \cup D_2)$,
  see Figure \ref{fig:Deet}.
  \begin{figure}
    \psfrag{x}{$y'$}
    \psfrag{y}{$y$}
    \psfrag{z}{$z$}
    \psfrag{t}{$\theta_j$}
    \psfrag{r1}{$r_1$}
    \psfrag{r2}{$3r_1$}
    \begin{center}
      \includegraphics[width=0.4\textwidth]{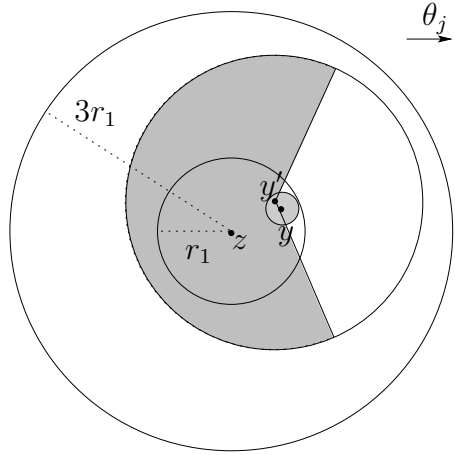}
    \end{center}
    \caption{The covering of the set $A \cap B(z,r_1)$ in the proof of
      Theorem \ref{thm:all}. The smallest ball is the set $D_1$ and
      the darkened sector is the set $D_2$. The rest of the set is
      called $D_3$.}
    \label{fig:Deet}
  \end{figure}
  Observe that $D_3 \cap A_j = \emptyset$ and so
  \eqref{eq:bestdirection} implies
  \begin{equation*}
    \mu(D_3) \le (1-K^{-1}) \mu\bigl( A \cap B(z,r_1) \bigr).
  \end{equation*}
  Moreover, the inequality \eqref{eq:rightconepoint} reads
  \begin{equation*}
    \mu(D_1) < c'c \mu\bigl( B(z,r_1) \bigr)
  \end{equation*}
  and with \eqref{eq:tripling} and the fact that $y'\in A_j$, we are
  able to conclude that
  \begin{equation*}
    \mu(D_2) < 2c \mu\bigl( B(y',2r_1) \bigr) \le 2c \mu\bigl( B(z,
    3r_1) \bigr) \le 2c'c \mu\bigl( B(z,r_1) \bigr).
  \end{equation*}
  Putting these three estimates together yields
  \begin{equation*}
    \mu\bigl( A \cap B(z,r_1) \bigr) \le 3 c'c \mu\bigl( B(z,r_1)
    \bigr) + (1-K^{-1}) \mu\bigl( A \cap B(z,r_1) \bigr)
  \end{equation*}
  from which we get
  \begin{equation*}
    \mu\bigl( A \cap B(z,r_1) \bigr) \le 3K c'c \mu\bigl( B(z, r_1)
    \bigr) = \tfrac{1}{3} \mu\bigl( B(z,r_1) \bigr).
  \end{equation*}
  This contradicts \eqref{eq:densitypoint} and finishes the proof.
\end{proof}

\section{Measures with positive Hausdorff dimension}
\label{sec:conical}

Suppose that $\HH_h$ is a Hausdorff
measure constructed using a non-decreasing gauge function
$h\colon(0,r_0)\rightarrow(0,\infty)$ and $\mu$ is its restriction to
some Borel set with finite $\HH_h$ measure. There are many works
(e.g.\ \cite{Salli1985}, \cite{Mattila1988},
\cite{KaenmakiSuomala2004}) that give information on the amount of
$\mu$ on small cones around $(n-m)$-planes $V\in G(n,n-m)$ when $h$
satisfies suitable assumptions. These results apply when $\HH^m$ is purely
singular with respect to $\HH_h$. In \cite{KaenmakiSuomala2008},
similar results are obtained also for many packing type measures. In
this section, we consider general measures with $\dimh(\mu)>m$ in the
same spirit by proving the following result.

\begin{theorem} \label{thm:dimh_positive}
  If $n \in \N$, $m \in \{ 0,\ldots,n-1 \}$, $s>m$, and $0<\alpha\le 1$, then there exists a constant $c=c(n,m,s,\alpha)>0$ so that for every measure $\mu$ on $\R^n$ with $\dimh(\mu) \ge s$ we have
  \begin{equation}\label{eq:claim}
    \limsup_{r \downarrow 0} \inf_{\atop{\theta \in S^{n-1}}{V \in
        G(n,n-m)}} \frac{\mu\bigl( X(x,r,V,\alpha) \setminus
      H(x,\theta,\alpha) \bigr)}{\mu\bigl( B(x,r) \bigr)} > c
  \end{equation}
  for $\mu$-almost every $x \in \R^n$.
\end{theorem}

We first introduce a couple of geometric lemmas. The first one is
proved by  Erd\H{o}s and F\"uredi in \cite{ErdosFuredi1983} with the
correct asymptotics for $q(n,\alpha)$ as $\alpha \to 0$. See
also \cite[Lemma 2.1]{KaenmakiSuomala2004}.

\begin{lemma} \label{thm:erdos}
  For each $0 < \alpha \le 1$ there is $q = q(n,\alpha) \in \N$
  such that in any set of $q$
  points in $\R^n$, there are always three points $x_0$, $x_1$, and
  $x_2$ for which $x_1 \in X^+(x_0,\theta,\alpha)$ and $x_2 \in
  X^+(x_0,-\theta,\alpha)$ for some $\theta \in S^{n-1}$.
\end{lemma}

We would like to apply the previous lemma for balls instead of just
single points. For this, we will need the following simple lemma.

\begin{lemma} \label{thm:etamato}
  For each $0 < \alpha \le 1$ there is $t = t(\alpha) \ge 1$ such that if $x_0,y_0 \in \R^n$ and $r_x,r_y > 0$
  are such that $B(x_0,tr_x) \cap B(y_0,tr_y) = \emptyset$ and $y_0
  \in X^+(x_0,\theta,\alpha/t)$ for some $\theta \in S^{n-1}$, then
  \begin{equation*}
    B(y_0,r_y) \subset X^+(x,\theta,\alpha)
  \end{equation*}
  for all $x \in B(x_0,r_x)$.
\end{lemma}

\begin{proof}
  Fix $y \in B(y_0,r_y)$ and $x \in B(x_0,r_x)$. Our aim is to find $t
  \ge 1$ depending only on $\alpha$ so that under the assumptions of
  the lemma we have
  \begin{equation*}
    (y-x) \cdot \theta > (1-\alpha^2)^{1/2}|y-x|.
  \end{equation*}
  Let $\eps = \bigl( 1-(1-\alpha^2)^{1/2} \bigr)/2$ and choose $t \ge
  1$ so large that $(1-(\alpha/t)^2)^{1/2} \ge 1-\eps$, $(1-\eps)t - 1
  > 0$, and $(1-\eps)/(1+1/t) - 1/(t+1) >
  (1-\alpha^2)^{1/2}$. According to our assumptions, we have
  \begin{align}
    |y_0-x_0| &\ge t(r_y+r_x), \label{eq:angle_2} \\
    (y_0-x_0) \cdot \theta &\ge (1-\eps)|y_0-x_0|. \label{eq:angle_1}
  \end{align}
  Also, we clearly have $(y-x) \cdot \theta \ge (y_0-x_0) \cdot \theta
  - (r_y+r_x) > 0$ and $|y-x| \le |y_0-x_0| + r_y+r_x$. Hence
  \begin{equation} \label{eq:angle_3}
    \frac{(y-x) \cdot \theta}{|y-x|} \ge
    \frac{(y_0-x_0) \cdot \theta}{|y_0-x_0| + r_y+r_x} -
    \frac{r_y+r_x}{|y_0-x_0| + r_y+r_x}.
  \end{equation}
  Now \eqref{eq:angle_2} yields
  \begin{equation*}
    \frac{r_y+r_x}{|y_0-x_0| + r_y+r_x} \le \frac{1}{t+1}
  \end{equation*}
  and by using \eqref{eq:angle_1} and \eqref{eq:angle_2}, we get
  \begin{equation*}
    \frac{|y_0-x_0| + r_y+r_x}{(y_0-x_0) \cdot \theta} \le
    \frac{1}{1-\eps} + \frac{1}{(1-\eps)t}.
  \end{equation*}
  The proof is finished by combining these estimates with
  \eqref{eq:angle_3} and the choice of $t$.
\end{proof}

The following somewhat technical proposition reduces the proof of Theorem
\ref{thm:dimh_positive} to finding a suitable amount of roughly
uniformly distributed balls inside $B(x,r)$ all having quite large
measure. If this can be done at arbitrarily small scales around
typical points, then Theorem \ref{thm:dimh_positive} follows. Below,
we shall denote by $\#\mathcal{B}$ the cardinality of a collection
$\mathcal{B}$.

\begin{remark} \label{rem:grass}
  Observe that $G=G(n,n-m)$ endowed with the metric $d(V,W) = \sup_{x \in V \cap S^{n-1}} \dist(x,W)$ is a compact metric space and
  \[\bigcup_{d(W,V)<\alpha}\{x\,:\,x\in W\}=X(0,V,\alpha)\]
  for all $V\in G$ and $0<\alpha<1$. See \cite[Lemma 2.2]{Salli1985}.
  Using the compactness, we may thus
  choose $K=K(n,m,\alpha)\in\N$ and $(n-m)$-planes $V_1,\ldots,V_K
  \in G$,
  so that for each $V \in G$ there is $j \in \{ 1,\ldots,K \}$ with
  \begin{equation} \label{eq:discrete_planes}
    X(x,V,\alpha) \supset X(x,V_j,\alpha/2)
  \end{equation}
  for all $x \in \R^n$.
\end{remark}

\begin{proposition} \label{thm:disjoint_balls}
Let $m\in\{0,\ldots,n-1\}$, $0<\alpha\leq 1$, $t=t(\alpha/2)$ be the constant of Lemma \ref{thm:etamato}, and $q=q(n-m,\alpha/(2t))$ from Lemma
\ref{thm:erdos}. Moreover, let $K=K(n,m,\alpha)$ be as in Remark \ref{rem:grass} and $c>0$. Suppose that $\mu$ is a
measure on $\R^n$ and that for $\mu$-almost all $x\in\R^n$ we may
find arbitrarily small radii $r>0$ and a collection $\mathcal{B}$ of sub-balls of $B(x,r)$ with the following properties:
  \begin{enumerate}
\item \label{assu1}   The collection $\{2tB : B\in\mathcal{B}\}$ is
  pairwise disjoint.
\item\label{assu2}    $\mu(B) > c\mu\bigl( B(x,3r) \bigr)$  for all
  $B\in\mathcal{B}$.
\item \label{assu3}   If
$\mathcal{B}'\subset\mathcal{B}$ with
$\#\mathcal{B}'\geq\#\mathcal{B}/K$ and $V\in G(n,n-m)$, then there is
a translate of $V$ intersecting at least $q$
  balls from the collection $\mathcal{B'}$.
  \end{enumerate}
Then
  \begin{equation}\label{eq:assu}
    \limsup_{r \downarrow 0} \inf_{\atop{\theta \in S^{n-1}}{V \in
        G(n,n-m)}} \frac{\mu\bigl( X(x,r,V,\alpha) \setminus
      H(x,\theta,\alpha) \bigr)}{\mu\bigl( B(x,r) \bigr)} > c
  \end{equation}
  for $\mu$-almost every $x \in \R^n$.
\end{proposition}

\begin{proof}
  Let $\mu$ be a measure satisfying the assumptions of the proposition and suppose that $(n-m)$-planes $V_1,\ldots,V_K$ are as in Remark \ref{rem:grass}.
  Our aim is to show that for $\mu$-almost every $x \in \R^n$, there
  are arbitrarily small radii $r>0$ so that for every $j \in \{
  1,\ldots,K \}$ there is $\zeta=\zeta(x)\in S^{n-1} \cap V_j$ for which
  \begin{equation} \label{eq:goal}
    \min\bigl\{ \mu\bigl( X^+(x,r,\zeta,\alpha/2) \bigr), \mu\bigl(
    X^+(x,r,-\zeta,\alpha/2) \bigr) \bigr\} > c\mu\bigl( B(x,r) \bigr).
  \end{equation}
  From this the claim follows easily. Indeed, take $V\in G(n,n-m)$ and
  choose $V_j\in\{V_1,\ldots,V_K\}$ so that \eqref{eq:discrete_planes}
  holds. Let $\zeta\in V_j\cap S^{n-1}$ satisfy \eqref{eq:goal}. Then
  \begin{equation*}
    X^+(x,r,\pm\zeta,\alpha/2) \subset X(x,r,V_j,\alpha/2) \subset X(x,r,V,\alpha)
  \end{equation*}
  and the claim follows by combining \eqref{eq:goal} with the observation
  that for all $\zeta',\theta\in S^{n-1}$ we have
  \begin{equation*}
    X^+(x,r,\zeta',\alpha) \cap H(x,\theta,\alpha) = \emptyset \text{ or }
    X^+(x,r,-\zeta',\alpha) \cap H(x,\theta,\alpha) = \emptyset.
  \end{equation*}

  To prove \eqref{eq:goal}, we assume on the contrary that there is a
  Borel set $F \subset \R^n$ with $\mu(F)>0$ such that the assumptions
  \eqref{assu1}--\eqref{assu3} hold for every $x
  \in F$ in some arbitrarily small scales and that for some $r_0 > 0$ and
  for every $0 < r < r_0$, there is
  $j \in \{ 1,\ldots,K \}$ so that
  \begin{equation} \label{eq:antithesis}
    \mu\bigl( X^+(x,r,\zeta,\alpha/2) \bigr) \le c\mu\bigl( B(x,r) \bigr)
    \text{ or }
    \mu\bigl( X^+(x,r,-\zeta,\alpha/2) \bigr) \le c\mu\bigl( B(x,r) \bigr)
  \end{equation}
  for all $\zeta\in S^{n-1}\cap V_j$.
  Now choose a $\mu$-density point $x_1$ of $F$
  and a radius $0<r_1<r_0/3$ so that
  \begin{equation} \label{eq:density_point}
    \mu\bigl( B(x_1,r) \setminus F \bigr) < c\mu\bigl(B(x_1,r)\bigr)\leq
    c\mu\bigl( B(x_1,3r) \bigr)
  \end{equation}
for all $0<r<r_1$. Next we choose a radius $0<r<r_1$ and a collection
of balls $\mathcal{B}$ inside $B(x_1,r)$ satisfying the assumptions
\eqref{assu1}--\eqref{assu3}.
 Then we let
\[F_j=\{x\in B(x_1,r)\cap F\,:\,\eqref{eq:antithesis}\text{ holds with
  this }r\text{ for all }\zeta\in S^{n-1}\cap V_j\}.\]
for $j\in\{1,\ldots,K\}$. According to \eqref{eq:density_point} each ball of $\mathcal{B}$ contains points of $F$ and hence there is at least one
$j\in\{1,\ldots,K\}$ so that not less than $\#\mathcal{B}/K$ balls among
$\mathcal{B}$ contain points of $F_j$. Fix such a $j$, and let
$\mathcal{B}'=\{B\in\mathcal{B} : F_j\cap B\neq\emptyset\}$. Then the
assumption \eqref{assu3} implies that we may find $z\in\R^n$ and $q$
different balls $B_1,\ldots,B_q\in\mathcal{B}'$ so that they all
intersect the affine $(n-m)$-plane $V_j+z$. According to the
assumption \eqref{assu1} and Lemmas \ref{thm:erdos} and
\ref{thm:etamato}, we may
choose three balls $B^0,B^1,B^2$ among the balls $B_1,\ldots,B_q$ and
a point $x_0\in F_j\cap B^0$ so
that for some $\theta \in S^{n-1} \cap V_j$ we have
  \begin{equation*}
    B^1 \subset X^+(x_0,\theta,\alpha/2) \text{ and } B^2 \subset X^+(x_0,-\theta,\alpha/2).
  \end{equation*}
  But this contradicts \eqref{eq:antithesis} since
    $\min\{ \mu(B^1),\mu(B^2) \} > c\mu\bigl( B(x_1,3r) \bigr) \ge c\mu\bigl( B(x_0,2r) \bigr)$ by the assumption \eqref{assu2}.
\end{proof}

To complete the proof of Theorem \ref{thm:dimh_positive}, we need to
find collections $\mathcal{B}$ of balls as in the previous
proposition. To that end, we first work with cubes (instead of balls)
and use Theorem \ref{thm:jarv}.

\begin{lemma} \label{thm:homolemma}
  For any $n \in \N$, $m \in \{ 0,\ldots,n-1 \}$, $s>m$, $M \in \N$, $\tau \ge 1$, and $k > M^{1/(s-m)}$ there exist constants $c = c(n,m,s,M,\tau,k)>0$ and $0 < p = p(n,m,s,M,\tau,k) < 1$ satisfying the following: For every measure $\mu$ on $[0,1)^n$ with $\dimh(\mu) \ge s$ and for $\mu$-almost every $x \in [0,1)^n$,
  \begin{equation} \label{eq:large_cubes}
  \begin{split}
    \limsup_{l \to \infty} \tfrac{1}{l}\#\bigl\{ j \in \{ 1,\ldots,l
    \} : \;&\mu(Q_{\iii,k^n-M k^m}) > c\mu(\tau Q_\iii) \text{ where} \\ &\iii \in I_k^j \text{ is such that } x \in Q_\iii \bigr\} > p.
  \end{split}
  \end{equation}
  Here we use the enumeration of the $k$-adic cubes introduced in \S \ref{sec:not}.
\end{lemma}

\begin{proof}
  Since $\log(M k^m)/\log(k)<s$, it follows by an easy calculation
  that we may choose a number $c=c(n,m,s,M,\tau,k)>0$ such that $0 <
  \eta := 3c (3\sqrt{n}\tau+2)^n < k^{-n}$ and
\begin{equation} \label{eq:c_def}
\begin{split}
  -\frac{1}{\log k} \biggl( &(k^n-M k^m) \eta\log\eta \\
  &+ \bigl(1-(k^n-M k^m)\eta\bigr) \log\biggl(\frac{1-(k^n-M k^m)\eta}{M k^m}\biggr)\biggr)<s.
\end{split}
\end{equation}
We will prove the claim with this choice of $c$, and with
$p=c(3\sqrt{n}\tau+2)^n$.
Suppose to the contrary that there is a Borel set $F \subset [0,1)^n$
with $\mu(F)>0$ such that \eqref{eq:large_cubes} does not hold for any
point of $F$.
Consider the restriction measure $\mu|_F$. In order to use Theorem
\ref{thm:jarv}, we scale our original measure so that $\mu(F)=1$. Note
that this
scaling does not affect the dimension of $\mu$ nor the condition
\eqref{eq:large_cubes}. It is enough to show that
\begin{equation}\label{eq:2homoeq}
  \hom_k^{k^n-Mk^m}(\mu|_F) \le 3ck^n(3\sqrt{n}\tau + 2)^n
\end{equation}
since this would imply $\dimh(\mu)\leq\dimh(\mu|_F)<s$ by Theorem \ref{thm:jarv} and \eqref{eq:c_def}. In order to calculate $\hom_k^{k^n-M k^m}(\mu|_F)$, we need to enumerate the $k$-adic cubes in terms of $\mu|_F$, not in terms of $\mu$. We denote cubes  enumerated in terms of $\mu|_F$ by $Q'_{\iii}$.

  Observe that if $Q \in \QQ_k^j$, then any ball centred at $Q$
  with radius $\sqrt{n}\tau k^{-j}$ contains the cube $\tau Q$ and is
  contained in the cube $3\sqrt{n}\tau Q$.
If $x\in F$ is a $\mu$-density point of $F$, then
$\mu\bigl( B(x,r) \bigr) \le 2\mu\bigl( F \cap B(x,r) \bigr)$ for all
$r>0$ small enough.
If $j \in \N$ is large enough and $\mu(Q_{\iii,k^n-M k^m}) \le
c\mu(\tau Q_\iii)$, where $\iii \in I_k^j$ is such that $x \in
Q_\iii$, then also
  \begin{equation*} 
  \begin{split}
    \mu(F \cap Q_{\iii,i}) &\le \mu(Q_{\iii,k^n-M k^m}) \le
    c\mu(\tau Q_\iii)
    \le c\mu\bigl( B(x,\sqrt{n}\tau k^{-j}) \bigr) \\
    &\le 2c\mu\bigl( F \cap B(x,\sqrt{n}\tau k^{-j}) \bigr)
    \le 2c\mu(F \cap 3\sqrt{n}\tau Q_\iii)
  \end{split}
  \end{equation*}
for all $i\in\{1,\ldots,k^n-M k^m\}$ and so also
\begin{equation}\label{eq:F_cubes}
 \mu(F \cap Q'_{\jjj,k^n-M k^m})\leq 2c\mu(F \cap 3\sqrt{n}\tau Q'_\jjj)
\end{equation}
where $Q'_{\jjj}=Q_{\iii}$.

We denote $E_k^j = \{ \iii \in I_k^j : \mu(F \cap Q'_{\iii,k^n-M k^m})
\le 2c\mu(F \cap 3\sqrt{n}\tau Q'_\iii) \}$ for $j \in \N$ and $N(x,l)
= \#\bigl\{ j \in \{ 1,\ldots,l \} : \iii|_j \in E_k^j \text{ where }
\iii \in I_k^l \text{ is such that } x \in Q'_\iii \bigr\}$ for $x \in
[0,1)^n$ and $l \in \N$. It follows from the choice of the set $F$ and
\eqref{eq:F_cubes} that
    $\liminf_{l \to \infty} \tfrac{1}{l} N(x,l) \ge 1-p$
  for $\mu$-almost every $x \in F$. Since $N(x,l)$ is constant on $Q'_\iii$ whenever $\iii \in I_k^l$, this implies
\begin{equation*}
    \liminf_{l \to \infty} \tfrac{1}{l} \sum_{j=1}^l \sum_{\iii \in
      E_k^j} \mu(F \cap Q'_\iii) = \liminf_{l \to \infty}
    \tfrac{1}{l} \int_F N(x,l) d\mu(x) \ge 1-p
\end{equation*}
  by Fatou's lemma, and consequently,
  \begin{equation*}
    \limsup_{l \to \infty} \tfrac{1}{l} \sum_{j=1}^l \sum_{\iii \notin
      E_k^j} \mu(F \cap Q'_\iii) \leq p.
  \end{equation*}
Moreover,
  \begin{equation*}
    \sum_{\iii \in I_k^j} \mu(F \cap 3\sqrt{n}\tau Q'_\iii) \le
    (3\sqrt{n}\tau + 2)^n
  \end{equation*}
  for every $j \in \N$, because each cube $Q \in \QQ_k^j$ intersects at most
  $(3\sqrt{n}\tau+2)^n$ larger cubes $3\sqrt{n}\tau
  \widetilde{Q}$ where $\widetilde{Q} \in \QQ_k^j$. Combining the previous two estimates and the choice of $p$, we now obtain
  \begin{align*}
    \hom_k^{k^n-M k^m}(\mu|_F)
    &= \limsup_{l \to \infty}
    \tfrac{k^n}{l} \sum_{j=1}^l \biggl( \sum_{\iii \in E_k^j}
    \mu(F \cap Q'_{\iii,k^n-M k^m})\\
    &\qquad\qquad\qquad\;\;\;
+ \sum_{\iii \notin E_k^j} \mu(F \cap
    Q'_{\iii,k^n-M k^m}) \biggr) \\
     &\le\limsup_{l \to \infty}
    \tfrac{k^n}{l} \sum_{j=1}^l \sum_{\iii \in I_k^j}
    2c\mu(F\cap 3\sqrt{n}\tau Q'_\iii) \\
 &\qquad\qquad\qquad\;\;\;
+\limsup_{l \to \infty}
    \tfrac{k^n}{l} \sum_{j=1}^l \sum_{\iii \notin E_{k}^j}
    \mu(F\cap Q'_{\iii,k^n-M k^m}) \\
    &\le 2ck^n(3\sqrt{n}\tau + 2)^n + pk^n\\
 &= 3ck^n\left(3\sqrt{n}\tau+2\right)^n.
  \end{align*}
This completes the proof.
\end{proof}

To finish the proof of Theorem \ref{thm:dimh_positive}, we just need to
combine the previous lemma and Proposition
\ref{thm:disjoint_balls}, and show how cubes may be replaced by
balls. We will choose the number of cubes $Q_{\iii,i}$ with $\mu(Q_{\iii,i}) > c\mu(\tau Q_\iii)$ (using the notation of Lemma \ref{thm:homolemma}) large enough so that we are able to choose sufficiently many appropriately separated balls $Q_{\iii,i} \subset B_i \subset \tau Q_\iii$. In order to find a ball containing $\tau Q_\iii$ with comparable measure, we need to work on a doubling scale $\iii$. For this, we will use Lemma \ref{thm:doubling_scales}.

\begin{proof}[Proof of Theorem \ref{thm:dimh_positive}.]
  Observe that without loss of generality, we may assume $\mu$ to be a probability measure with $\spt(\mu) \subset [0,1)^n$.
  Let $t=t(\alpha/2) \ge 1$ be the constant of Lemma \ref{thm:etamato} and $q=q(n-m,\alpha/(2t))$ from Lemma \ref{thm:erdos}. Moreover, let
  $K=K(n,m,\alpha)$ be as in Remark \ref{rem:grass} and choose $M=M(n,m,\alpha)\in\N$ so that
$M\geq \vol(n)(4t+2)^n n^{n/2} 8^m K q $,
where $\vol(n)$ is the $n$-dimensional volume of the
 unit ball.

If $Q \in
  \QQ_k^j$ for some $j,k \in \N$ and $\tau = 6\sqrt{n}$, it follows that
  \begin{gather}
    2Q \subset B(x,2\sqrt{n}k^{-j}) \subset \tau Q, \label{eq:ball_cube1} \\
    B(y,\sqrt{n}k^{-j-1}) \subset B(x,2\sqrt{n}k^{-j}) \label{eq:ball_cube2}
  \end{gather}
  for every $x,y \in Q$. Choose $k=k(n,m,s,\alpha)\in\N$ so that
  $k>\max\{ M^{1/(s-m)},3 \}$ and let $c_1
  =c(n,m,s,M,\tau,k)> 0$ and $0 <
  p = p(n,m,s,M,\tau,k) < 1$ be as in Lemma \ref{thm:homolemma} and
  $c_2 = c(n,k,1-p/2) > 0$ be the constant of Lemma
  \ref{thm:doubling_scales}. Combining these lemmas it follows that
  for $\mu$-almost all $x \in [0,1)^n$
  there are arbitrarily large $j \in \N$ and $\iii \in I_k^j$ with
  $x \in Q_\iii$ such that with $r = 2\sqrt{n}k^{-j}$ we have
  \begin{align}
 \mu\bigl( B(x,r) \bigr) &\ge c_2\mu\bigl( B(x,2\sqrt{n}k^{-j+1})
 \bigr), \label{eq:muest1}\\
 \mu(Q_{\iii,k^n-M k^m}) &> c_1\mu(\tau Q_\iii). \label{eq:muest2}
  \end{align}
  To obtain \eqref{eq:muest1}, we use Lemma \ref{thm:doubling_scales}
  with $\gamma=2\sqrt{n}$.
  To complete the proof, the only thing to check is that with any such $x$ and $r$ we may find a collection $\mathcal{B}$ satisfying the assumptions
  \eqref{assu1}--\eqref{assu3} of Proposition \ref{thm:disjoint_balls}.

  Combining \eqref{eq:muest2}, \eqref{eq:ball_cube1}, and \eqref{eq:muest1} and recalling that $k \geq 3$, we have
  \begin{equation} \label{eq:big_cubes}
    \mu(Q_{\iii,i}) > c_1\mu\bigl( B(x,r) \bigr) \ge c_1 c_2 \mu\bigl(
    B(x,3r) \bigr)
  \end{equation}
  for every $i \in \{ k^n-M k^m,\ldots,k^n \}$. Let $B_i =
  B(y_i,\sqrt{n}k^{-j-1})$ where $y_i$ is the centre point of $Q_{\iii,i}$. Then
  $\mu(B_i) > c_1c_2 \mu\bigl( B(x,3r) \bigr)$ and $B_i \subset
  B(x,r)$ by \eqref{eq:ball_cube2}. By a simple volume argument, we
  have $\#\{ j : 2tB_i \cap 2t B_j \ne \emptyset \} \le
  \vol(n)(4t+2)^n n^{n/2}$
 for every $i$. Consequently, there is a sub-collection $\mathcal{B}$ of the
 collection $\{B_i\}$ containing at least
  $8^m Kq k^m$ balls so that
 the collection $\{2tB : B\in\mathcal{B}\}$ is
  pairwise disjoint and $\mu(B) > c_1 c_2\mu\bigl( B(x,3r) \bigr)$  for all
  $B\in\mathcal{B}$.
 To check that also the assumption \eqref{assu3} of Proposition
 \ref{thm:disjoint_balls} holds, choose any sub-collection
 $\mathcal{B}'$ of $\mathcal{B}$ with
 $\#\mathcal{B}'\geq\#\mathcal{B}/K\geq 8^m q k^m$ and fix $V\in
 G(n,n-m)$. Since the $m$-dimensional ball
 $\proj_{V^\bot}\bigl(B(x,r)\bigr)$ may be covered by $8^m
 k^m$ balls
   of radius $\sqrt{n} k^{-j-1}$, it follows that some translate of $V$ must
   hit at least $q$ balls from the collection $\mathcal{B}'$. Here $V^\bot$ denotes the orthogonal complement of $V$. Thus we
   have verified the assumptions of Proposition
   \ref{thm:disjoint_balls} and the claim follows with
   $c=c(n,m,s,\alpha)=c_1 c_2$.
\end{proof}

\begin{remark} \label{rem:quantitative}
(1) Our method to prove Theorem \ref{thm:dimh_positive} could
be pushed further to obtain the following quantitative upper conical
density theorem: Under the assumptions of Theorem \ref{thm:dimh_positive}, we have
\begin{equation*}
\limsup_{l\rightarrow\infty} \tfrac1l \#\bigl\{ j \in \{1,\ldots,l\} : \inf_{\atop{\theta \in S^{n-1}}{V \in
        G(n,n-m)}} \frac{\mu\bigl( X(x,2^{-j},V,\alpha)\bigr) \setminus
      H(x,\theta,\alpha) \bigr)}{\mu\bigl( B(x,2^{-j})\bigr)} > c \bigr\} > p
\end{equation*}
for $\mu$-almost all points $x\in\R^n$ with some constants
$c=c(\alpha,s,n,m)>0$ and $p=p(\alpha,s,n,m)>0$.

(2) One could also apply Mattila's result \cite[Theorem
3.1]{Mattila1988} to obtain results analogous to Theorem
\ref{thm:dimh_positive}. More precisely, the quantity
\[\inf\limits_{\atop{\theta \in S^{n-1}}{V \in
        G(n,n-m)}} \frac{\mu\bigl(X(x,r,V,\alpha)\setminus
        H(x,\theta,\alpha)\bigr)}{\mu\bigl(B(x,r)\bigr)}\] can be
      replaced by
      \[\inf_C \frac{\mu\bigl(C_x\cap
        B(x,r)\bigr)}{\mu\bigl(B(x,r)\bigr)},\] where the infimum is
      over all Borel sets $C\subset G(n,n-m)$ with
      $\gamma(C)>\delta>0$. Here $C_x=\bigcup_{V\in C}(V+x)$, and $\gamma$ is the natural isometry
      invariant Borel probability measure on the Grasmannian
      $G(n,n-m)$. The obtained constant $c>0$ then depends on
      $n$, $m$, $s$, and $\delta$.

Thus, using Mattila's method would yield more general results in the
sense that the cones $X(x,V,\alpha)$ could be replaced by the more
general cones $C_x$. On the other hand, our method allows to
consider also the non-symmetric cones $X(x,V,\alpha)\setminus
H(x,\theta,\alpha)$ and may be used to obtain quantitative estimates
as in Remark \ref{rem:quantitative}(1).
\end{remark}

\section{Examples and open problems} \label{sec:ex}

Inspecting the proof of Proposition \ref{thm:disjoint_balls}, we see that
the assumptions of Theorem \ref{thm:dimh_positive} imply that we may,
in fact, find directions $\theta_{x,V}\in S^{n-1}\cap V$, depending on
the point $x$, such that
\begin{equation} \label{eq:claim2}
  \limsup_{r \downarrow 0} \inf_{V \in G(n,n-m)}
      \frac{\min\bigl\{\mu\bigl(X^+(x,r,\theta_{x,V},\alpha)\bigr), \mu\bigl(X^+(x,r,-\theta_{x,V},\alpha) \bigr)\bigr\}}{\mu\bigl(B(x,r)\bigr)} > c
\end{equation}
for $\mu$-almost all $x\in\R^n$. If $m=0$, we do not know if the
assumption $\dimh(\mu)>0$ is necessary or not:

\begin{question} \label{question}
Given $\alpha>0$ and $n\in\N$, does there exist a constant
$c(n,\alpha)>0$ so that for all non-atomic measures $\mu$ on $\R^n$ one
could pick $\theta=\theta(x)\in S^{n-1}$ for $\mu$-almost all
$x\in\R^n$ so that
\[\limsup_{r\downarrow 0} \frac{\min\bigl\{\mu\bigl(X^+(x,r,\theta,\alpha)\bigr),
\mu\bigl(X^+(x,r,-\theta,\alpha)\bigr)\bigr\}}{\mu\bigl(B(x,r)\bigr)}>c\,?\]
\end{question}

\begin{remark} \label{rem:question}
(1) A positive answer would also improve Theorem \ref{thm:all}.
However, the question is relevant only for $n\geq2$. If $n=1$, there
is no difference between the above question and Theorem \ref{thm:all}.

(2) Examples \ref{ex:Pertti} and \ref{ex:LastExampl} below
show that we cannot hope to
obtain \eqref{eq:claim} if the dimension of $\mu$ is $m$, even if
$\mu$ is purely unrectifiable (see the definition before Example \ref{ex:LastExampl}). Thus, Question \ref{question} is really
only about non-atomic measures with zero Hausdorff dimension.
\end{remark}

The following example shows why we cannot apply Proposition
\ref{thm:disjoint_balls} to answer Question \ref{question}. For
simplicity, we will work on $\R$, although similar construction works
also in higher dimensions.

\begin{example}
  There is a non-atomic measure $\mu$ on $\R$ so that it fails to
  satisfy the assumptions of Proposition \ref{thm:disjoint_balls} with
  $m=0$ for all $c>0$.
\end{example}

\begin{proof}[Construction.]
We will construct the measure $\mu$ on $[0,1)$. Our aim is to show that there is no constant $c>0$ so that for $\mu$-almost all $x \in [0,1)$ there would be arbitrarily small radii $r>0$ such that we could find intervals $I_1,\ldots,I_6 \subset (x-r,x+r)$ for which
  \begin{gather}
    3I_i \cap 3I_j = \emptyset \text{ whenever } i \ne j, \label{eq:3Idisjoint} \\
    \mu(I_i) > c\mu(x-3r,x+3r) \text{ for all } i. \label{eq:mubig}
  \end{gather}
  To construct $\mu$, we simply take any sequence $0 < q_i < 1/2$ so that $\sum_{i=1}^{\infty} q_i = \infty$ and $q_i \downarrow 0$ as $i \rightarrow \infty$. Then we construct a binomial type measure using the weights $q_i$ and $p_i = 1-q_i$. Let $\mu([0,1/2)) = p_1$ and $\mu([1/2,1)) = q_1$. If $i \in \N$ and $J \in \QQ_2^i$, then for $I_1,I_2 \in \QQ_2^{i+1}$, where $I_1 \subset J$ is the left-hand side subinterval and $I_2 \subset J$ is the right-hand side subinterval, we set $\mu(I_1) = p_{i+1}\mu(J)$ and $\mu(I_2) = q_{i+1}\mu(J)$. This construction extends to a measure by standard methods.

  Suppose there is a constant $c>0$ for which \eqref{eq:3Idisjoint} and \eqref{eq:mubig} hold. Choose $i_0 \in \N$ so that
  \begin{equation} \label{eq:qi}
    q_i < c/3 \quad \text{for all } i > i_0.
  \end{equation}
  We may assume that \eqref{eq:3Idisjoint} and \eqref{eq:mubig} are valid for $I_1,\ldots,I_6 \subset I := (x-r,x+r) \subset [0,1)$ with $r \ll 2^{-i_0}$. Choose $l \in \N$ for which $2^{-l-1} \le 2r < 2^{-l}$. Then $I$ intersects at most three dyadic intervals of length $2^{-l-1}$ and one of these dyadic intervals, say $J$ must contain at least two of the intervals $I_1,\ldots,I_6$, say $I_1$ and $I_2$. Now $J \subset 3I$ so $\mu(I_1),\mu(I_2) > c\mu(3I) \ge c\mu(J)$.

  Let $J_0 \subset J$ be the largest dyadic sub-interval of $J$ with the same left-hand side end point as $J$ for which
  \begin{equation} \label{eq:J1small}
    \mu(J_0) < c\mu(J).
  \end{equation}
  Let $y$ be the right-hand side end point of $J_0$ and let $J_1,\ldots,J_k$ be the maximal dyadic sub-intervals of $J$ which do not intersect $J_0$. So $J = J_0 \cup J_1 \cup \cdots \cup J_k$ and $J_i \cap J_j = \emptyset$ whenever $i \ne j$. It follows from the construction of $\mu$ and \eqref{eq:qi} that $\mu(J_i) \le \tfrac{c}{3} \mu(J)$ for all $i \ge 1$. So if $y \notin I_1$, then $I_1 \cap J_0 = \emptyset$ by \eqref{eq:J1small}, and $I_1$ has to intersect at least three of the intervals $J_1,\ldots,J_k$. Then $J_i \subset I_1$ for at least one $i \ge 1$. Since $J_0 \subset 3J_i$ for all $i$ it follows that also $J_0 \subset 3I_1$. In particular $y \in 3I_1$, in any case. By the same argument also $y \in 3I_2$, so $3I_1 \cap 3I_2 \neq \emptyset$ contrary to \eqref{eq:3Idisjoint}.

  Observe that one may replace $3$ in \eqref{eq:mubig} by any number
  $a > 1$, but then $6$ (the number of the chosen sub-intervals) needs
  to be replaced by $n=n(a) \in \N$.
\end{proof}


To finish the paper, we give the examples mentioned in Remark \ref{rem:question}(2).
Suppose that $A\subset\R^n$ is purely $m$-unrectifiable and satisfies
$0<\HH^m(A)<\infty$.
We refer the reader to \cite{Mattila1995} for the basic properties of unrectifiable sets.
If $0<\alpha<1$ and $V\in G(n,n-m)$, it is well known that
\begin{equation}\label{eq:unrect}
\limsup_{r\downarrow 0} \frac{\HH^m\bigl(A \cap X(x,r,V,\alpha)\bigr)}{(2r)^m} > c(m,\alpha) > 0
\end{equation}
for $\HH^m$-almost all $x\in A$. The following example, answering
\cite[Question 4.2]{KaenmakiSuomala2008}, shows that this
cannot be improved to
\begin{equation*}
\limsup_{r\downarrow 0} \inf_{V \in G(n,n-m)} \frac{\HH^m\bigl(A \cap X(x,r,V,\alpha)\bigr)}{(2r)^m} > c(m,\alpha) > 0.
\end{equation*}

\begin{example} \label{ex:Pertti}
There exists a purely $1$-unrectifiable compact set $A \subset \R^2$ with $0 < \HH^1(A) < \infty$ so that for every $0 < \alpha \le 1$
\begin{equation}\label{eq:examplecondition}
 \lim_{r \downarrow 0} \inf_{\ell \in G(2,1)} \frac{\HH^1\bigl(A \cap X(x,r,\ell,\alpha)\bigr)}{2r} = 0
\end{equation}
for every $x \in A$.
\end{example}

\begin{proof}[Construction.]
We construct the set $A$ using a nested sequence of compact sets. The
first set $A_0$ is just the unit ball $B(0,1)$. To define the rest of the construction sets, we apply the ideas found e.g.\ in \cite[\S 5.3]{MartinMattila1988} and \cite[\S 5.8]{Preiss1987}.

Define a collection of mappings $f_{i,j}$ with $i\in \N$ and $j \in \{1, \dots, 2i^2\}$ as
\begin{align*}
f_{i,j}(x,y)=\frac{1}{2i^2}\big( & (\cos(\alpha_i)x+2j-2i^2-1)-(-1)^j\sin(\alpha_i)y, \\
& (-1)^j\sin(\alpha_i)x+\cos(\alpha_i)y \big),
\end{align*}
where $\alpha_i = 1/\sqrt{i}$. Then define sets $A_n$ for $n \in \N$, as
\[
 A_n = \bigcup_{\substack{i\in \{1,\dots,n\}\\ j_i\in\{1,\dots,2i^2\}}} f_{1,j_1}\circ \cdots \circ f_{n,j_n} (A_0).
\]
Finally, set $A=\bigcap_{n=1}^\infty A_n$. See Figure \ref{fig:balls} to see the first three steps, $A_0$, $A_1$, and $A_2$, of the construction. We refer to the radius of step $n$ construction ball as $R_n$. That is $R_0 =1$ and $R_n = \frac{R_{n-1}}{2n^2}$ for $n\ge 1$.
\begin{figure}
\psfrag{a}{$\alpha_1$}
\psfrag{b}{$\alpha_2$}
\centering
\includegraphics[width=0.5\textwidth]{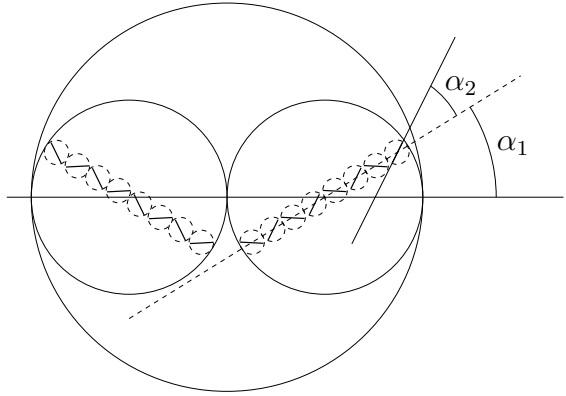}
\caption{An illustration for the construction of the set $A$ in Example \ref{ex:Pertti}.}
\label{fig:balls}
\end{figure}

Let us verify that the set $A$ admits the desired properties. It is
evident from the construction that $A\subset B(0,1)$ is a compact set
with $0<\HH^1(A)\leq 1$. The upper bound is trivial as the sum of the diameters of level $n$ construction balls is always one. If $F \subset B(0,1)$, then there is $n$ and a collection $\BB$ of level $n$ construction balls covering $F \cap A$ so that $\sum_{B \in \BB} \diam(B) < 10\diam(F)$. This gives the lower bound. Moreover, we have $\HH^1(A\cap
B_n)=R_n\HH^1(A)$ for each construction ball $B_n$ of level $n$.
For each $x \in A$ there is a unique address $a(x) =
\bigl(a_1(x),a_2(x),\ldots\bigr)$ so that  $a_i(x) \in \{1,\dots, 2i^2\}$
and
\[
 \{ x \} = \bigcap_{i=1}^\infty f_{1, a_1(x)} \circ \cdots \circ f_{i,a_i(x)} (A_0).
\]
By Kolmogorov's zero-one law and the three-series criteria (for
example, see \cite{Loeve1963}), the series $\sum_{i=1}^n(-1)^{a_i(x)}\alpha_i$ diverges for $\HH^1$-almost every $x \in A$. Take such a point $x$ and fix an angle $\beta \in [0, 2\pi]$. Since $\alpha_i \downarrow 0$ as $i \to \infty$, there is $\eps > 0$ so that
\[
 \limsup_{n \to \infty} \min_{k\in \Z}|\beta - \sum_{i=1}^n(-1)^{a_i(x)}\alpha_i + k\pi| > 4\eps.
\]
Let $\ell_\beta$ be the line with an angle $\beta$. We will
show that
\begin{equation}\label{eq:notapprtan}
\limsup_{r \downarrow 0} \frac{\HH^1\bigl(A \cap B(x,r) \setminus X(x,\ell_\beta,\eps)\bigr)}{r} > 0.
\end{equation}
This means that $\ell_\beta$ is not an approximate tangent of $A$ at
$x$ and thus $A$ is purely $1$-unrectifiable, see for example
\cite[Corollary 15.20]{Mattila1995}. Take $n \in \N$ large enough so that
\[
\min_{k\in \Z}|\beta - \sum_{i=1}^n(-1)^{a_i(x)}\alpha_i+ k\pi| > 2\eps.
\]
Since all the $2n^2$ level $n$ construction balls inside the ball
$f_{1,a_1(x)}\circ \cdots \circ f_{n-1,a_{n-1}(x)}(A_0)$ hit the line
from $x$ with direction $\sum_{i=1}^n(-1)^{a_i(x)}\alpha_i$, there
exists $K$ depending only on $\eps$ (it suffices to take
$K>10/\eps$) so that
\[
\#\{m : B_m \cap X(x,R_{n-1},\ell_\beta,\eps) \ne \emptyset\} \le K,
\]
where $B_m = f_{1,a_1(x)}\circ \cdots \circ f_{n-1,a_{n-1}(x)} \circ f_{n,m} (A_0)$. This yields an adequate surplus of balls outside the cone $X(x,\ell_\beta,\eps)$ giving
\[
 \frac{\HH^1\bigl(A\cap B(x,R_{n-1}) \setminus X(x,\ell_\beta,\eps)\bigr)}{R_{n-1}} \ge \frac{2n^2-K}{2n^2}\HH^1(A)
\]
and therefore \eqref{eq:notapprtan} holds.

It remains to verify \eqref{eq:examplecondition} holds. Let $x\in A$
and $0 < \alpha \le 1$. First observe from the construction that with
any $n \in \N$ and $y \in A \setminus \bigl(f_{1,a_1(x)}\circ \cdots \circ
f_{n-1,a_{n-1}(x)}(A_0)\bigr)$ we have
\[
\dist(y,x)\ge \bigl(1-\cos(\alpha_{n})\bigr)R_{n-1} \ge R_{n-1}/(4n) =
2n^2 R_n/(4n) = n R_n/2.
\]
Let $0<r<1$ and choose the $n\in \N$ for which $nR_n \leq
2r<(n-1)R_{n-1}$. Let $\ell$ be the line perpendicular to the direction
$\sum_{i=1}^{n-1}(-1)^{a_i(x)}\alpha_i$. Now there are numbers
$M,n_0\in N$ depending only on $\alpha$ (letting $M>10/\alpha$ and
$n_0$ so that $\alpha_{n_0-1}<\alpha/10$ will do) so that
if $n\geq n_0$, then
\[
\#\{m : B_m \cap X(x,r,\ell,\alpha) \ne \emptyset\} \le M,
\]
where $B_m$'s denote the construction balls of level $n$. Thus
\[
\frac{\HH^1\bigl(A \cap X(x,r,\ell,\alpha)\bigr)}{2r} \le \frac{M R_n \HH^1(A)}{nR_n} = \frac{M}{n} \HH^1(A) \longrightarrow 0,
\]
as $r\downarrow 0$.
\end{proof}

A measure $\mu$ on $\R^n$ is called \emph{purely $m$-unrectifiable} if
$\mu(A)=0$ for all $m$-rectifiable sets $A\subset\R^n$.
The following example shows that a result analogous to
\eqref{eq:unrect} does not hold for arbitrary purely $m$-unrectifiable
measures on $\R^m$.

\begin{example}\label{ex:LastExampl} There exists $\ell \in G(2,1)$ and a measure $\mu$ on $\R^2$ so that $\mu$ is purely $1$-unrectifiable and for every $0<\alpha<1$
\begin{equation}\label{eq:conicaldensityzero}
 \lim_{r\downarrow 0} \frac{\mu\bigl(X(x,r,\ell,\alpha)\bigr)}{\mu\bigl(B(x,r)\bigr)} = 0
\end{equation}
for $\mu$-almost all $x \in \R^2$.
\end{example}

\begin{proof}[Construction.]
We construct the measure $\mu$ using families of maps
\[
  \bigl\{f_{k,h}^i : k \in \{ 0,\ldots,i-1 \} \text{ and } h \in \{ 0,\ldots,2i^2-1 \} \bigr\}_{i=1}^\infty
\]
with
\[
 f_{k,h}^i\bigl((x,y)\bigr) = \biggl(\frac{(-1)^ki+x}{2i^3},\frac{2ki^2+h+y}{2i^3}\biggr)
\]
for every $i \in \{2,3,\dots\}$, $ k \in \{0, \dots, i-1\}$ and $h \in \{0, \dots, 2i^2-1\}$.

With $\{f_{k,h}^i\}_{k,h}$ define $F_i$ that maps a measure $\nu$ on $\R^2$ to a measure $F_i(\nu)$ so that for every Borel set $A \subset \R^2$ we get
\begin{equation}\label{eq:massdistribution}
 F_i(\nu)(A) = \sum_{k = 0}^{i-1} \sum_{h=0}^{2i^2-1} C_i(2i)^{-|h-i^2+\frac{1}{2}|}\nu\bigl((f_{k,h}^i)^{-1}(A)\bigr),
\end{equation}
where the constant $C_i$ is chosen so that $\sum_{k = 0}^{i-1} \sum_{h=0}^{2i^2-1} C_i(2i)^{-|h-i^2+\frac{1}{2}|} = 1$. Applying the map $F_i$ divides the measure into $i$ vertical strips. These strips correspond to the index $k$ in the mappings $f_{k,h}^i$. Inside the strips the measure is divided to $2i^2$ blocks using the index $h$. The measure is concentrated near the centres of the strips by giving different weights to the maps $f_{k,h}^i$ with different values of $h$. See Figure \ref{fig:lasexample} to get the idea of the distribution of mass under map $F_i$.
\begin{figure}
\psfrag{f}{$F_2$}
\centering
\includegraphics[width=0.8\textwidth]{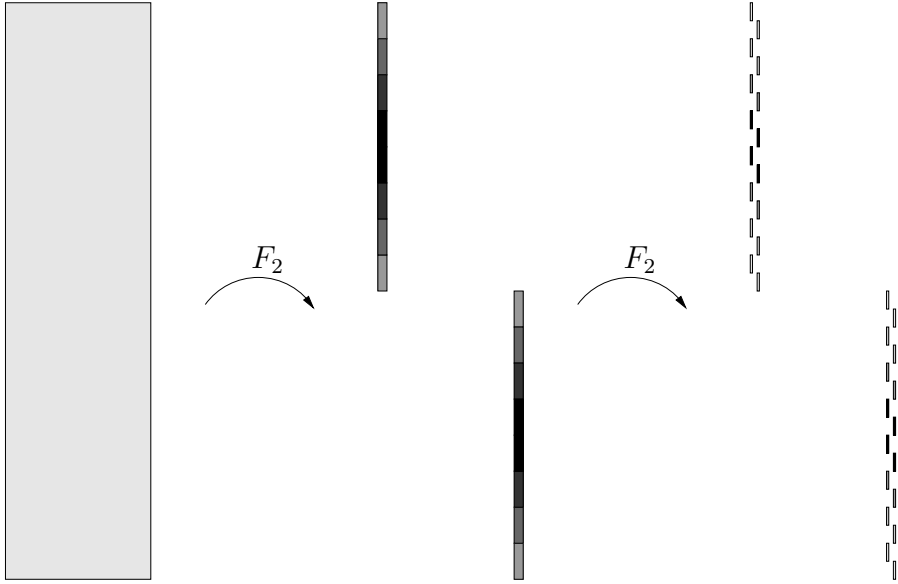}
\caption{The distribution of the measure with map $F_2$ in Example \ref{ex:LastExampl}.}
\label{fig:lasexample}
\end{figure}

Let $N_1=0$ and for $i \in \{ 2,3,\ldots \}$ let $N_i$ be the smallest integer so that
\begin{equation}\label{eq:numberofscales}
 \biggl(1-\frac{C_i}{8(2i)^{i^2-\frac{3}{2}}}\biggr)^{N_i} < \frac{1}{2}.
\end{equation}
Integers $N_i$ determine how many times we have to use map $F_i$ when constructing the measure $\mu$ in order to make the resulting measure unrectifiable. With these numbers define $(I_j)_{j=1}^\infty$ with
\[
 I_{p+\sum_{i=1}^{t-1}N_i}=t
\]
for every $t \in \{2,3,\dots\}$ and $p\in \{1,\dots,N_t\}$. Also let $M_j = \prod_{i=1}^j(2I_i^3)$.
Finally define $\mu$ to be the weak limit of
\[
 F_{I_1}\circ F_{I_2} \circ \cdots \circ F_{I_m}(\mu_0)
\]
as $m \to \infty$. Here $\mu_0$ is any compactly supported Borel probability measure on $\R^2$. (Take for example $\HH^1$ restricted to $\{0\}\times[0,1]$.) With $i\in \N$, $k\in\{1,\dots, M_{i-1}I_i\}$ and $h\in\{1,\dots,2I_i^2\}$ define strips
\[
 S_{i,k} = \spt(\mu) \cap \biggl( \R \times \biggl[\frac{2(k-1)I_i^2}{M_i},\frac{2kI_i^2}{M_i}\biggr] \biggr)
\]
and blocks
\[
 B_{i,k,h} = \spt(\mu) \cap \biggl(\R \times \biggl[\frac{2(k-1)I_i^2+h-1}{M_i},\frac{2kI_i^2+h}{M_i}\biggr]\biggr).
\]

To prove the unrectifiability, let us first look at vertical curves: Let $\gamma$ be a $C^1$-curve in $\R^2$ so that $|\frac{\partial\gamma}{\partial y}| \ge\frac{1}{3}|\gamma'|$. Take $i \in \N$. Now for any $k \in \{1, \dots, I_{i+1}-1\}$ and $t \in \{0,\dots,M_i-1\}$  either
\[
 \gamma \cap B_{i+1,2I_{i+1}^3t+k,2I_{i+1}^2} = \emptyset \text{ or }\gamma \cap B_{i+1,2I_{i+1}^3t+k+1,1} = \emptyset.
\]
This means that when we look at two consecutive strips $S_{i+1,2I_{i+1}^3t+k}$ and $S_{i+1,2I_{i+1}^3t+k+1}$, we see that the curve $\gamma$ cannot meet both the uppermost block of the lower strip and the lowest block of the upper strip. This is because vertically these blocks are next to each other, but horizontally the distance is roughly at least $I_{i+1}$ times the width of the block. Hence the curve $\gamma$ misses more than one fourth of all the end blocks of the strips of the level $I_{i+1}$ construction step. Therefore by iterating and using inequality \eqref{eq:numberofscales}, we get
\begin{align*}
 \mu(\gamma) & \le \prod_{i=1}^M \biggl(1-\frac{I_{i+1}C_{I_{i+1}}(2I_{i+1})^{-I_{i+1}^2+\frac{1}{2}}}{4} \biggr)\\
 &\le \prod_{m=2}^{I_M-1}\biggl(1-\frac{C_m}{8(2m)^{m^2-\frac{3}{2}}}\biggr)^{N_m} < 2^{-I_M+2} \to 0
\end{align*}
as $M \to \infty$.

Next we look at horizontal curves: Let $\gamma$ be a $C^1$-curve in $\R^2$ so that $|\frac{\partial\gamma}{\partial x}| \ge\frac{1}{3}|\gamma'|$. Take $i \in \N$ and $t \in \{0,\dots,M_i-1\}$. Now there are at most two $k \in \{1,\dots, I_{i+1}\}$ so that
\[
 \gamma \cap S_{i+1,tI_{i+2}+k} \ne \emptyset.
\]
By repeating this observation
\[
 \mu(\gamma) \le \prod_{i=2}^M\frac{2}{I_i} \to 0
\]
as $M \to \infty$. Take any $C^1$-curve $\gamma$ in $\R^2$. Because it can be covered with a countable collection of vertical and horizontal $C^1$-curves defined as above, we have $\mu(\gamma) = 0$. Thus, the measure $\mu$ is purely $1$-unrectifiable.

Let $\ell\in G(2,1)$ be the horizontal line. We show that cones around $\ell$ have small measure in the sense of equality \eqref{eq:conicaldensityzero}. To do this fix $0 < \alpha < 1$ and take the smallest $i_0 \in \{3,4,\dots\}$ so that
\begin{equation}\label{eq:anglegivingextrablock}
 \frac{1}{I_{i_0}} < \frac{\sqrt{1-\alpha}}{4}.
\end{equation}
Now take $i \in \{i_0+1,i_0+2,\dots\}$, a point $x \in \spt(\mu)$ and a radius $r \in [M_i^{-1},M_{i-1}^{-1}]$. Let $k_1 \in \N$ so that $x \in S_{i,k_1}$. Assume that there are at most two  $k' \in \N$ so that
\[
 X(x,r,\ell,\alpha) \cap S_{i+1,k'} \ne \emptyset.
\]
Then
\begin{equation}\label{eq:coneestimate1}
  \mu\bigl(X(x,r,\ell,\alpha)\bigr) \le \frac{2\mu\bigl(B(x,r)\bigr)}{I_{i+1}}.
\end{equation}

Assume then that there are at least three such $k'$. If this is the case, then the cone $X(x,r,\ell,\alpha)$ must hit another large vertical strip $S_{i,k_2}$ with $k_2 \in \{k_1 - 1, k_1+1\}$.
Inequality \eqref{eq:anglegivingextrablock} yields the existence of a block $B_{i,k_1,u} \subset B(x,r)$ whose vertical distance to the centre of the strip $S_{i,k_1}$ is strictly less than the vertical distance from the centre of the strip $S_{i,k_2}$ to any of the blocks $B_{i,k_2,u'}$ that intersect the cone $X(x,r,\ell,\alpha)$. Now the fact that we concentrated measure to the centre using equation \eqref{eq:massdistribution} gives
\[
 \mu(B_{i,k_1,u})\ge \frac{(2I_i)^u}{2\sum_{p=1}^{u-1}(2I_i)^p}\mu\bigl(X(x,r,\ell,\alpha)\bigr)
\]
and hence
\[
 \mu\bigl(X(x,r,\ell,\alpha)\bigr) \le \frac{2\mu\bigl(B(x,r)\bigr)}{I_i}.
\]
This together with \eqref{eq:coneestimate1} shows \eqref{eq:conicaldensityzero} as $i$ tends to infinity.
\end{proof}


\end{document}